\documentclass[reqno, makeidx]{amsart}
\usepackage{amsmath, mathrsfs}
\usepackage{latexsym}
\usepackage{amsfonts}
\usepackage{amssymb}
\usepackage{amsthm}
\usepackage{amsbsy}
\usepackage{graphicx}
\usepackage{esint}
\usepackage{epsf}
\usepackage{color}
\usepackage{epstopdf}
\usepackage[all]{xy}

\newcommand{\R}{\mathbb{R}}

\newcommand{\rn}{\mathbb{R}^N}

\newcommand{\G}{\mathbb{G}}

\newcommand{\cH}{\mathscr{H}}
\newcommand{\cg}{\mathrm{g}}

\newcommand{\Om}{\Omega}
\newcommand{\nh}{\nabla_\cH}

\newcommand{\Hn}{\mathbb H^n}
\newcommand{\p}{\partial}
\newcommand{\la}{\lambda}
\newcommand{\ve}{\varepsilon}
\newcommand{\Z}{\mathscr Z}
\newcommand{\vf}{\varphi}

\newtheorem{defi}{Definition}[section]
\newtheorem{lemma}{Lemma}[section]
\newtheorem{teo}{Theorem}[section]

\newtheorem{rem}{Remark}[section]
\newtheorem{prop}{Proposition}[section]

\begin{document}

\title{Starshapedeness for fully-nonlinear equations in Carnot groups}


\author{Federica Dragoni}
\address{School of Mathematics\\
Cardiff University\\
Cardiff CF2 4AG
WALES} \email[Federica Dragoni]{DragoniF@cardiff.ac.uk}

\author{Nicola Garofalo}
\address{Dipartimento di Ingegneria Civile, Edile e Ambientale (DICEA) \\ Universit\`a di Padova\\ 35131 Padova, ITALY}
\email[Nicola Garofalo]{nicola.garofalo@unipd.it}

\author{Paolo Salani}
\address{Department of Mathematics and Computer Science\\
Universit\`a di Firenze\\
50134 Florence,
ITALY}\email[Paolo Salani]{paolo.salani@unifi.it}

\thanks{The first author was partially supported by the EPSRC Grant ``Random Perturbations of ultra-parabolic PDEs under rescaling''. \\
The second author was supported in part by a Progetto SID (Investimento Strategico di Dipartimento) ``Non-local operators in geometry and in free boundary problems, and their connection with the applied sciences", University of Padova, 2017.\\
The third author was partially supported by GNAMPA of INdAM and by a ``Progetto Strategico di Ateneo'', University of Firenze.
}

\maketitle


\begin{abstract}
In this paper we establish the starshapedness of the level sets of the capacitary potential of a large class of fully-nonlinear equations for condensers in Carnot groups. 
\end{abstract}



 \section{Introduction}\label{S:intro} 

The study of the geometric properties of solutions to elliptic and parabolic problems such as starshapedness, convexity and symmetry properties is a classical subject of investigation. For instance, it is well-known that the capacitary potential of a starshaped ring in $\rn$ has level sets which are all starshaped. The objective of this paper is to establish a similar result for the capacitary potential of a wide class of fully-nonlinear equations and for condensers in Carnot groups, once a natural notion of starshapedness has been introduced. Our main result, Theorem \ref{teoremaPuntoGenerico}, encompasses operators as diverse as:
\begin{itemize}
\item the horizontal Laplacean $\Delta_\cH$ in \eqref{hlap} below;
\item the horizontal $q$-Laplacean $\Delta_{\cH,q}$ in \eqref{hplap} with $1<q<\infty$;
\item the horizontal $\infty$-Laplacean $\Delta_{\cH,\infty}$ in \eqref{infty},
\end{itemize}
just to name a few basic examples.

\medskip

To introduce the relevant geometric framework we consider a stratified nilpotent Lie group $(\G,\circ)$ of step $k\in \mathbb N$ with  group law $\circ$. This means that $\G$ is a simply-connected real Lie group whose Lie algebra  $\cg$ admits a stratification 
\[
\mathrm{g}=\mathrm{g}_1\oplus \dots\oplus\mathrm{g}_k
\]
 which is $k$-nilpotent, see Definition \ref{CarnotgroupDefi} in Section \ref{S:cg} below. The first layer $\cg_1$ of the Lie algebra is called the \emph{horizontal} layer. It plays a special role since it bracket-generates the whole Lie algebra.  
Such groups naturally arise as approximating ``tangent spaces'' to the sub-Riemannian spaces associated with a H\"ormander type operator in $\rn$
\[
\mathscr L = \sum_{i,j=1}^m X_j^2 + X_0,
\]
where $X_0, X_1,...,X_m$ are $C^\infty$ vector fields satisfying the finite rank condition on the Lie algebra in \cite{Ho}. This guiding idea was at the basis of a visionary program which E. Stein presented in his address at the 1970 International Congress of Mathematicians in Nice \cite{Snice}, and that culminated in the celebrated joint works with L. Rotchschild \cite{RS} on the \emph{lifting theorem}, and with Nagel and Wainger on the size of metric balls and fundamental solution estimates, see \cite{NSW}. 
Stratified nilpotent Lie groups have also been known as \emph{Carnot groups} after the foundational work of Gromov, see e.g. \cite{Gro1} and \cite{Gro2}, and Pansu, see \cite{P}.  In this paper, we will routinely use this latter denomination.

\medskip

Given a non-Abelian Carnot group $(\G,\circ)$ (this means that we assume that its step is $k>1$), one can use the left-translations $L_p(p') = p\circ p'$ or the right-translations $R_p(p') = p'\circ p$ to move about $\G$ and to introduce differential operators. We will restrict the attention to the family of left-translations $\{L_p\}_{p\in \G}$, and indicate with $\cH\subset T\G$ the corresponding subbundle of the tangent bundle whose fiber at a point $p\in \G$ is given by $\cH_p = dL_p(\cg_1)$, where $dL_p$ indicates the differential of $L_p:\G\to \G$.  $\cH$ is called the \emph{horizontal} bundle. 
With a given orthonormal basis $\{e_1,...,e_m\}$ of $\mathrm{g}_1$ we can associate corresponding $C^\infty$ left-invariant vector fields $\{X_1,...,X_m\}$ by the formula
\begin{equation}\label{xj}
X_j(p) = dL_p(e_j),\ \ \ \ j=1,...,m.
\end{equation}
By (i) and (ii) in Definition \ref{CarnotgroupDefi} below the $X_j$'s and their commutators of order up to $k$ generate the whole Lie algebra of left-invariant vector fields on $\G$. Given a function $f\in C^\infty(\G)$, we denote by 
\begin{equation}\label{hg}
\nabla_\cH f = \sum_{j=1}^m X_j f X_j
\end{equation}
its \emph{horizontal gradient}. We note that \eqref{hg} represents the projection of the Riemannian gradient $\nabla f$ onto the horizontal bundle $\mathscr H$. We henceforth formally identify $\nh f$  with the $m$-dimensional vector $(X_1 f,...,X_m f)^T\in\R^m$.
We also indicate with 
\begin{equation}\label{hessian}
\nabla^2_\cH f =  \left[\frac{X_i X_j f + X_j X_i f}{2}\right]_{i,j=1,...,m}
\end{equation}
the symmetrized \emph{horizontal Hessian}  of $f$ with respect to the orthonormal basis $\{e_1,...,e_m\}$ of $\cg_1$. 

\medskip

In addition to the left-translations $\{L_p\}_{p\in \G}$ a Carnot group is also equipped with a one-parameter family of non-isotropic dilations $\{\delta_\lambda\}_{\lambda>0}$ which at the level of the Lie algebra are prescribed by formally assigning the weight $j$ to the $j$-th layer $\cg_j$ in its stratification, see \eqref{dilg} and \eqref{dilG} below. Using the dilations $\{\delta_\lambda\}_{\lambda>0}$ and the left-translations $\{L_p\}_{p\in \G}$ the following notion of sharshapedness in a Carnot group $\G$ was introduced in \cite{D-G} and \cite{DG2}.

\begin{defi}\label{D:star}
Let $\Om\subset \G$ be a connected set containing the group identity $e$. We say that $\Om$ is \emph{starshaped} (with respect to $e$) if
\[
\Om = \underset{0\le \la \le 1}{\bigcup} \delta_\la(\Om).
\]
\end{defi}

Given $p_0\in \Om$, we say that $\Om$ is starshaped with respect to $p_0$ if $L_{p_0^{-1}}(\Om)$ is starshaped. This is equivalent to saying that
\[
\Om = \underset{0 \le \la \le 1}{\bigcup} L_{p_0}\left(\delta_\la\left(L_{p_0^{-1}}(\Om)\right)\right),
\]
i.e., for every $p\in \Om$ and $0 \le \lambda \leq 1$ one must have
\begin{equation}\label{Zgo}
\delta^{p_0}_\lambda(p) \overset{def}{=} p_0 \circ \delta_\lambda \left(p_0^{-1} \circ p\right) \in \Om.
\end{equation}

We call $\{\delta^{p_0}_\lambda\}_{\lambda>0}$ the dilations centered at the point $p_0$. It is worth emphasizing here that although this definition is consistent both with the standard notion of starshapedness in $\R^N$ and with Pansu's notion in \cite{P} of differentiability at a point $p_0$ of a map between Carnot groups $f:\G\to \G'$, some surprising and unsettling phenomena can occur. For instance, in the unit gauge ball centered at $e$ in the Heisenberg group $\Hn$,
\[
B_1 = \{(z,t)\in \Hn\mid |z|^4 + 16 t^2 < 1\},
\]
there exists a continuum of points $\{p_0(\ve)\mid 0<\ve<\ve_0\}$  such that, for every fixed $\ve \in (0,\ve_0)$, $B_1$ is \emph{not} starshaped with respect to $p_0(\ve)$, see Proposition \ref{P:gaugeb} below. This negative phenomenon of course has no equal in the Euclidean $\R^N$, where every convex set (and therefore any Euclidean ball) is starshaped with respect to any of its points.

 Given a Carnot group $\G$, we consider a \emph{condenser} $(\Omega_0,\overline \Omega_1)$,  where $\Omega_i\subset \G$, $i=0,1$, are connected bounded open sets such that 
\[
\overline\Omega_1\subset\Omega_0.
\]
We henceforth indicate with $\Omega = \Omega_0\setminus \overline \Omega_1$ the ring between $\Om_0$ and $\Om_1$. 
We indicate with $\mathscr S_m(\R)$ the space of all $m\times m$ symmetric matrices with real coefficients, and we consider a function 
\[
\mathscr F:
\overline \Om_0\times \R\times \R^m\times \mathscr S_m(\R) \to \R,
\]
satisfying the following assumptions:
\begin{itemize}
\item[(i)] $\mathscr F$ is \emph{proper}, which means that given $p\in \Om, \xi\in \R^m$ and $M\in \mathscr S_m(\R)$ one has
\begin{equation}\label{proper}
\mathscr F(p,r,\xi,M)\leq \mathscr F(p,s,\xi,M), \quad \forall\, r\geq s;
\end{equation}
\item[(ii)]  $\mathscr F$ is \emph{degenerate elliptic}, which means that given $p\in \Om, r\in \R$ and $\xi\in \R^m$ one has
\begin{equation}\label{DegenerateElliptic}
\mathscr F(p,r,\xi,M)\leq \mathscr F(p,r,\xi,N), \; \forall\, M, N\in \mathscr S_m(\mathbb R) \ \text{such that}\  M\le N.
\end{equation}
\item[(iii)] $\mathscr F$ is \emph{stable under the scalings centered at the point $p_0\in\Omega_1$},
  in the sense that
\begin{equation}\label{structuralAssumptionF}
 \mathscr F(p,r,\xi,M)\geq 0
 \Rightarrow
\quad
\mathscr F\left(
\big(\delta_{\la}^{p_0}\big)^{-1}(p),r,{\la}\,\xi,{\la^2}M\right)\geq 0
\end{equation}
for all $(p,r,\xi,M)$ and $\la\geq 1$ such that $p\in\Omega$ and $\big(\delta^{p_0}_\la\big)^{-1}(p)\in\Omega$, 
where $\big(\delta_{\la}^{p_0}\big)^{-1}$ is the inverse of $\delta_{\la}^{p_0}$ (see \eqref{Zgo}),  i.e.
 $\big(\delta_{\la}^{p_0}\big)^{-1}(p)=p_0\circ \delta_{\frac{1}{\lambda}}\big(p_0^{-1}\circ p\big)$.
\end{itemize}
 \begin{rem}
Note that assumption (iii) is for example satisfied whenever 
$$
 F\left(
\big(\delta_{\la}^{p_0}\big)^{-1}(p),r,{\la}\,\xi,{\la^2}M\right)
=
\lambda^{\alpha}F(p,r,\xi,M)
$$
for some $\alpha\geq 0$.
\end{rem}

The purpose of the present work is to investigate the geometry of viscosity solutions to the following fully nonlinear capacitary problem
\begin{equation}\label{HorCapacity_intro}
\begin{cases}
\mathscr F\big(p,f,\nh f, \nh^2 f\big) = 0\ \ \ \quad\ \  \ \ \ \ \ \ \ \ \text{in}\ \Omega,
\\
0\leq f\leq 1\ \ \ \ \ \ \  \ \ \ \ \ \ \ \ \ \ \ \ \ \ \ \ \ \ \ \ \ \ \ \ \ \ \ \ \ \text{in}\ \Omega,
\\
f=0\ \ \ \ \ \ \ \ \ \ \ \ \ \ \ \ \ \ \ \ \ \ \ \ \ \ \ \ \ \ \ \ \ \ \ \   \quad\ \ \  \text{on}\ \partial \Omega_0,
\\
f=1\ \ \ \ \ \ \ \ \ \ \ \ \ \ \ \ \ \ \ \ \ \ \ \ \ \ \ \ \ \ \ \ \ \ \ \ \  \quad\ \  \text{in}\ \overline{\Omega}_1.
\end{cases}
\end{equation}
For the notion of viscosity solution to \eqref{HorCapacity_intro} we refer the reader to Definition \ref{Viscosity solutions} below.
Hereafter, we assume that $f$ be defined in the whole of $\overline{\Om}_0$ after we extend it with $1$ in $\overline{\Om}_1$.
In addition to the hypothesis \eqref{proper}, \eqref{DegenerateElliptic} and \eqref{structuralAssumptionF} above we will assume that the following comparison principle be valid.

\medskip

\noindent 
\emph{Let $f, g \in C(\overline \Om)$ respectively be a viscosity subsolution and a viscosity supersolution of \eqref{HorCapacity_intro}
. Then},
 \begin{equation}\label{CP}
 f\leq g\ \ \textrm{on}\ \ \partial \Omega\quad \Longrightarrow\quad 
 f\leq g\ \ \textrm{in}\ \ \overline{ \Omega}.
\end{equation}
Note that, whenever constant functions are solutions of the pde in \eqref{HorCapacity_intro}, then the comparison principle implies the maximum principle, which gives in particular that 
$
0\leq f\leq 1 $ on $\overline{ \Omega}$ under the given boundary conditions.

\medskip

The following theorem is the main result in this paper.

\begin{teo}[Starshapedness] \label{teoremaPuntoGenerico} 
Let $\G$ be a Carnot group and let $(\Om_0,\overline \Om_1)$ be a condenser in $\G$. Assume that $\Omega_i$, $i = 0, 1$, be starshaped with respect to a point $p_0\in \Om_1$. 
Suppose that  $\mathscr F :\overline \Om_0\times \R\times \R^m\times \mathscr S_m\to \R$ 
satisfy the hypothesis \eqref{proper},
\eqref{DegenerateElliptic}, \eqref{structuralAssumptionF} and \eqref{CP} above, and that $f\in C(\overline{\Omega})$ be a viscosity solution of problem \eqref{HorCapacity_intro}. Then, for any $\ell\in (0,1)$ the superlevel set
\[
\Omega_\ell= \{p\in \overline{\Om}_0\,|\, f(p) \geq \ell\}
\]
is starshaped with respect to $p_0$.
\end{teo}

As we have mentioned in the opening of the section Theorem \ref{teoremaPuntoGenerico} encompasses various pde's of interest in the analysis and geometry of Carnot groups. For a discussion of these models we refer the reader to Section \ref{S:models}. Here we note that, although Theorem \ref{teoremaPuntoGenerico} generalizes to Carnot groups a result by the third named author in $\R^N$, see Theorem 5.1 in \cite{PS}, the fact that starshapedness is preserved in the intricate geometry of Carnot groups is a perhaps unexpected and interesting phenomenon. One reason is that, while the fully nonlinear pde in \eqref{HorCapacity_intro} involves exclusively differentiation along the horizontal directions, the infinitesimal generator of the non-isotropic group dilations is not a horizontal vector field. In fact, it involves differentiation along all layers of the Lie algebra, see Proposition \ref{P:Z} below.

\medskip

We mention that, still in $\rn$, previous results in the direction of Theorem \ref{teoremaPuntoGenerico} are contained in the following papers \cite{Ge}, \cite{W}, \cite{BK}, \cite{L}, \cite{DK}, \cite{A}, \cite{Fr},  \cite{FG}, \cite{GS}, \cite{GR}, and in the more recent \cite{JKS}, where nonlocal operators are considered. In the framework of the present paper the only previous result connected to Theorem \ref{teoremaPuntoGenerico} is due to Danielli and the second named author. They considered in \cite{D-G} the exterior capacitary potential $f$ of a $C^2$, starshaped bounded set $\Om$ in a Carnot group of Heisenberg type $\G$
\begin{equation}\label{dg}
\begin{cases}
\Delta_\cH f = 0\ \ \ \ \ \ \ \ \ \ \text{in}\ \G\setminus \overline \Om,
\\
f\big|_{\p \Om} = 1,\ \ \ \ \ \ \underset{d(p,e)\to \infty}{\lim} f(p) = 0.
\end{cases}
\end{equation}
For such $f$ it was proved that for every $0<\ell<1$ the superlevel set of $f$,
\[
\Om_\ell = \{p\in \G\setminus \overline \Om\mid f(p) \ge \ell\}, 
\]
is strictly starshaped. This implies, in particular, that $\p \Om_\ell$ is a $C^\infty$ compact manifold of codimension one. A generalization of this result to arbitrary Carnot groups is contained in \cite{DG2}.

\medskip

This paper is organized as follows. 
In Section \ref{S:cg} we introduce some basic facts about Carnot groups which are needed for this paper. In Section \ref{S:models} we analyze some specific models of geometric pde's to which Theorem \ref{teoremaPuntoGenerico} applies. In Section \ref{S:star} we discuss the notion of starshapedness of a set and investigate some associated geometrical properties.
Finally, in Section \ref{S:proof} we prove Theorem \ref{teoremaPuntoGenerico}.


\section{Carnot groups}\label{S:cg}

In this section we introduce the definition of a Carnot group and recall some of its most basic aspects. For most of the properties listed in the sequel we refer the reader to Folland's 1975 seminal article \cite{Fo}, as well as to the books \cite{FS}, \cite{CGr}, and the more recent monographs \cite{Mon}, \cite{BLU} and \cite{Gems}.

\begin{defi}\label{CarnotgroupDefi}
Given $k\in \mathbb N$, a \emph{Carnot group} of step $k$ is a simply-connected real Lie group $(\G, \circ)$ whose Lie algebra $\cg$ is stratified and $k$-nilpotent. This means that there exist vector spaces $\mathrm{g}_1,...,\mathrm{g}_k$ such that  
\begin{itemize}
\item[(i)] $\mathrm{g}=\mathrm{g}_1\oplus \dots\oplus\mathrm{g}_k$;
\item[(ii)] $[\cg_1,\cg_j] = \cg_{j+1}$, $j=1,...,k-1,\ \ \ [\cg_1,\cg_k] = \{0\}$.
\end{itemize}
\end{defi}

Since according to (ii) the first layer $\cg_1$ of the Lie algebra plays a special role, it is called the \emph{horizontal}, or bracket-generating, layer of $\cg$. If $m_i = \operatorname{dim} \cg_i$, then we indicate with $N = m_1+\dots+m_k$ the topological dimension of the group $\G$. For notational simplicity we will always denote with $m=m_1$ the dimension of the horizontal layer $\cg_1$.

\medskip

 When $k=1$ the Lie group is Abelian, and we are back to the Euclidean setting. Since we are not interested in this case, in this paper we assume that $k\geq 2$. By the assumption that $\G$ be simply-connected we know that the exponential mapping $\exp: \cg \to \G$ is a global real-analytic diffeomorphism onto, see \cite{V} and \cite{CGr}. We will use this global chart to identify the point $p = \exp u\in \G$ with its logarithmic preimage $u\in \cg$. 

\medskip

Once the bracket relations at the level of the Lie algebra are assigned, the group law is too. This follows from the Baker-Campbell-Hausdorff formula, 
see, e.g., sec. 2.15 in \cite{V},
\begin{equation}\label{BCH}
\exp(u) \circ \exp(v) = \exp{\bigg(u + v + \frac{1}{2}
[u,v] + \frac{1}{12} \big\{[u,[u,v]] -
[u,[u,v]]\big\} + ...\bigg)},
\end{equation}
where the dots indicate commutators of order four and higher.  Furthermore, since by (ii) in \eqref{CarnotgroupDefi} above all commutators of order $k$ and higher are trivial, in every Carnot group the Baker-Campbell-Hausdorff series in the right-hand side of \eqref{BCH} is finite. 

\medskip

Every Carnot group is naturally equipped with a one-parameter family of automorphisms $\{\delta_\lambda\}_{\lambda>0}$  which are called the group dilations. One first defines a family of non-isotropic dilations $\Delta_\lambda :\cg \to \cg$ in the Lie algebra by assigning the formal degree $j$ to the $j$-th layer $\cg_j$ in the stratification of $\cg$. This means that if $u = u_1 + ... + u_k \in \cg$, with $u_j\in \cg_j$, $j = 1,...,k,$ one lets
\begin{equation}\label{dilg}
\Delta_\lambda u = \Delta_\lambda (u_1 + \cdots  + u_k) = \lambda u_1 + \cdots +  \lambda^k u_k.
\end{equation}
One then uses the exponential mapping to push forward \eqref{dilg} to the group $\G$, i.e., we define a one-parameter family of group automorphisms $\delta_\la :\G \to \G$ by the equation
\begin{equation}\label{dilG}
\delta_\lambda(p) = \exp \circ \Delta_\lambda \circ \exp^{-1}(p),\quad\quad\quad p\in \G.
\end{equation}
From the definition \eqref{dilG}, and the properties of $\Delta_\la$ on $\cg$, one easily verifies the following result.

\begin{lemma}\label{DilationsProperties}
For all  $\la, \mu >0$ one has:
\begin{enumerate}
\item[(1)]  $\delta_{1} = \operatorname{id}$;
\item[(2)] $\delta^{-1}_{\la} = \delta_{\la^{-1}}$;
  \item[(3)]
  $\delta_{\la}\circ \delta_{\mu} =\delta_{\la \mu}$;
\item[(4)] for every $p, p'\in \G$ one has $\delta_\la(p) \circ \delta_\la(p') = \delta_\la(p\circ p')$.
\end{enumerate}
\end{lemma}

\medskip

The \emph{homogeneous dimension} of $\G$ with respect to the dilations \eqref{dilG} is defined as follows
\begin{equation}\label{QG}
Q = \sum_{i = 1}^k i \operatorname{dim} \cg_i.
\end{equation}
Since we are assuming that $k>1$ this number is strictly bigger than the topological dimension $N$ of $\G$ and it plays a pervasive role in the analysis and geometry of $\G$. Given the one-parameter family of dilations \eqref{dilG} in $\G$, we introduce the following definition.

\begin{defi}\label{D:hom}
Let $\kappa\in \R$. A function $f:\G \to \R$ is called \emph{homogeneous} of degree $\kappa$ if for every $p\in \G$ and every $\la >0$ one has
\[
f(\delta_\la p) = \la^\kappa f(p).
\]
A differential operator $Y$ on $\G$ is called homogeneous of degree $\kappa$ if for every $f\in C^\infty(\G)$ one has
\[
Y(\delta_\la f) = \la^\kappa \delta_\la(Yf),
\]
where we have let $\delta_\la f(p) \overset{def}{=} f(\delta_\la p)$.
\end{defi} 

Concerning the vector fields $\{X_1,...,X_m\}$ in $\G$ associated to an orthonormal basis $\{e_1,...,e_m\}$ as in \eqref{xj} above, we have the following simple fact, see \cite{Fo}.

\begin{lemma}\label{L:xjhom}
For every $j = 1,...,m$ the vector field $X_j$ in \eqref{xj} is homogeneous of degree $\kappa = 1$, i.e., for any $f\in C^\infty(\G)$ one has
\[
X_j(\delta_\la f) = \la \delta_\la(X_j f).
\]
\end{lemma}

As a consequence of Lemma \ref{L:xjhom}
the horizontal Hessian $\nh^2 f$ in \eqref{hessian} above is homogeneous of degree two with respect to the dilations $\delta_\lambda$, i.e., for every $f\in C^\infty(\G)$ we have
\begin{equation}\label{laphom}
\nh^2 (\delta_\lambda f) = \lambda^2 \delta_\lambda(\nh^2 f).
\end{equation}
 
\medskip


\section{Some relevant geometric pde's}\label{S:models}

Among many possible examples, in this section we introduce some relevant model partial differential operators for the problem \eqref{HorCapacity_intro} and discuss them in connection with the hypothesis \eqref{proper},
\eqref{DegenerateElliptic}, \eqref{structuralAssumptionF} and \eqref{CP} in Theorem \ref{teoremaPuntoGenerico}. This is an important aspect since it puts our main result on a solid ground by showing that it is not empty, and that in fact it does apply to a variety of situations of interest in analysis and geometry.

\subsection{The horizontal Laplacean}\label{SS:hl}
The most basic example to keep in mind for the problem \eqref{HorCapacity_intro} is when
\begin{equation}\label{hlap}
\mathscr F\big(p,f,\nh f, \nh^2 f\big) = \operatorname{trace}(\nh^2 f) = \Delta_\cH f = \sum_{j=1}^m X_j^2 f.
\end{equation}
The linear partial differential operator in \eqref{hlap} is known as the \emph{horizontal Laplacean} (or sub-Laplacean) associated with the orthonormal basis $\{e_1,...,e_m\}$ of $\mathrm{g}_1$, and it plays a central role in analysis and sub-Riemannian geometry. We note here that, since we are assuming that the step of $\G$ is $k>1$, the operator $\Delta_\cH$ fails to be elliptic at every point $p\in \G$. However, thanks to H\"ormander's celebrated result in \cite{Ho}, in a Carnot group every horizontal Laplacean is hypoelliptic. This means that distributional solutions of $\Delta_\cH f = F$ are $C^\infty$ wherever such is $F$. By an adaptation of the Perron-Wiener-Brelot method for any condenser $(\Om_0,\overline \Om_1)$ there exists a viscosity solution $f\in C^\infty(\Om)$ to the problem \eqref{HorCapacity_intro}. Such $f$ represents the capacitary potential of the condenser. This function is important in the applications since, as it was shown in \cite{CDG}, it controls from above and below the asymptotic behavior of the Green's function of $\Delta_\cH$ near its singularity.

\medskip

We emphasize that, contrary to what happens in the standard Euclidean setting, even if the two sets $\Om_0$ and $\Om_1$ have $C^\infty$ boundary, it is not necessarily true in general that the assumption $f\in C(\overline{\Omega})$ in Theorem \ref{teoremaPuntoGenerico} be fulfilled. For instance, it was shown in Theorem 3.6 in \cite{HH} that when the group $\G$ has step $k\ge 3$ then the gauge balls (which have $C^\infty$ boundary) in \eqref{gb} are not necessarily regular for the Dirichlet problem for the horizontal Laplacean $\Delta_\cH$. 

\medskip

However, the hypothesis $f\in C(\overline{\Omega})$ does hold for large classes of condensers. For instance, a sufficient (purely metric) condition is that the domains $\Om_i$, $i = 0, 1$, be non-tangentially accessible (NTA) in the sense of \cite{CG}. This follows from the fact that NTA domains admit both an exterior and an interior non-tangential metric ball at every boundary point. From this property and the Wiener type criterion in \cite{NS} it follows that the Perron solution (and therefore the viscosity solution) to the problem \eqref{HorCapacity_intro} is in $C(\overline \Om)$. 

\medskip

Having said this, it is clear that the model \eqref{hlap}
 verifies the hypothesis \eqref{proper},
\eqref{DegenerateElliptic}, \eqref{structuralAssumptionF} (this latter hypothesis follows from Lemma \ref{L:xjhom} and \eqref{laphom}). Finally, the assumption \eqref{CP} is a direct consequence of Bony's strong maximum principle in \cite{Bony}.

\subsection{The horizontal $q$-Laplacean}\label{SS:hql}
More in general, given $1<q<\infty$ we consider the quasilinear operator, known as the \emph{horizontal $q$-Laplacean}, defined by
\begin{equation}\label{hplap}
\mathscr F\big(p,f,\nh f, \nh^2 f\big) = \Delta_{\cH,q} f \overset{def}{=} \sum_{j=1}^m X_j(|\nh f|^{q-2} X_j f).
\end{equation}
Operators such as \eqref{hplap} first (implicitly) appeared in 
Mostow's celebrated work on rigidity \cite{Mo}, and were subsequently studied by several other authors, see \cite{KR1}, \cite{KR2},  \cite{KR3},
\cite{P}, \cite{P2}, \cite{HR}, \cite{HH}, \cite{MM}, \cite{CDGcpde} and \cite{CDG}. In analysis \eqref{hplap} is connected to the Euler-Lagrange equation of the Folland-Stein horizontal Sobolev embedding $\mathscr S^{1,q}(\G) \subset L^{q^\star}(\G)$, when $1<q<Q$, and $1/q-1/q^\star = 1/Q$, with $Q$ as in \eqref{QG}.  
Notice that, if we define for $\sigma\in \R^m\setminus\{0\}$
 \begin{equation}\label{aij}
 a_{ij}(\sigma) = |\sigma|^{q-2}\left\{\delta_{ij} + (q-2) \frac{\sigma_i \sigma_j}{|\sigma|^2}\right\},
 \end{equation}
 then with $A(\sigma) = [a_{ij}(\sigma)]$ we can write the right-hand side of \eqref{hplap} in nondivergence form as follows
 \begin{equation}\label{hplap2}
 \mathscr F\big(p,f,\nh f, \nh^2 f\big) = \mathscr F\big(\nh f, \nh^2 f\big) = \operatorname{trace}(A(\nh f) \nh^2 f). 
 \end{equation}

\medskip

Given an arbitrary condenser $(\Om_0,\overline \Om_1)$ the existence of a weak solution to the problem \eqref{HorCapacity_intro} follows from the results in \cite{D} and \cite{TW}. Now, as a special case of the results in \cite{CDGcpde} one knows that weak solutions to $\Delta_{\cH,q} f = 0$ are locally in a H\"older class $C^{0,\alpha}$. On the other hand, in his papers \cite{tom} (Heisenberg group) and \cite{Bieske_P} (arbitrary Carnot groups) Bieske introduced the notion of viscosity solutions to $\Delta_{\cH,q} f = 0$, and he proved that they coincide with the weak solutions. Combining these results we conclude that for the operator \eqref{hplap} viscosity solutions of \eqref{HorCapacity_intro} are in fact at least $C_{loc}^{0,\alpha}$ (we mention that, in the Heisenberg group $\Hn$, Mukherjee and Zhong have recently proved in \cite{MZ} the remarkable (optimal) result that weak solutions of $\Delta_{\cH,q} f = 0$ are locally in the class $C^{1,\alpha}$, for some $\alpha = \alpha(n,q)\in (0,1)$.  
 An analogous optimal smoothness result for general Carnot groups presently remains an open question).
 
\medskip

Concerning the assumption $f\in C(\overline \Om)$ in Theorem \ref{teoremaPuntoGenerico}, we distinguish two cases: 1) $1<q\le Q$; and 2) $Q<q<\infty$. In case 1), similarly to Section \ref{SS:hl}, the validity of $f\in C(\overline \Om)$ is not guaranteed even when $\Om_i\in C^\infty$, $i=0,1$. Analogously to the linear case $q=2$ discussed in Section \ref{SS:hl}, a sufficient geometric assumption is that $\Om_i$ be NTA domains. Then, combining the existence of an exterior and interior non-tangential ball at every $p\in \p \Om$ with Theorems 3.1 and 3.9  in \cite{D} one concludes that every boundary point is regular for the Dirichlet problem, and thus the viscosity capacitary potential $f$ is in $C(\overline \Om)$. We mention here that when the group $\G$ has step $k = 2$, then every $C^{1,1}$ domain is NTA (and there exist $C^{1,\alpha}$ domains which are not regular for the Dirichlet problem), see \cite{CG} and \cite{MM}. In case 2), then in view of Theorem 8.1 in \cite{CDG} the $q$-capacity of a point is positive (bounded from below away from zero by a universal constant). By this observation and the Wiener type criterion in Theorem 6.2 in \cite{TW} we infer that, when $q>Q$, \emph{every bounded open set is regular for the (weak) Dirichlet problem}. Since by the results in the cited paper \cite{Bieske_P} weak solutions coincide with viscosity solutions, we conclude that for the operator \eqref{hplap} the assumption $f\in C(\overline{\Omega})$ in Theorem \ref{teoremaPuntoGenerico} is fulfilled for \emph{any} condenser.

Finally, the operator \eqref{hplap}
does verify the hypothesis \eqref{proper},
\eqref{DegenerateElliptic}, \eqref{structuralAssumptionF} (again, this follows from Lemma \ref{L:xjhom} and \eqref{laphom}), whereas the comparison principle for viscosity solutions in assumption \eqref{CP} is 
proved in \cite{{Bieske_P}} (Theorem 5.1).

\subsection{The horizontal $\infty$-Laplacean}\label{SS:infty} 
 
Yet another important model for Theorem \ref{teoremaPuntoGenerico} is the \emph{horizontal $\infty$-Laplacean} defined by
 \begin{equation}\label{infty}
\mathscr F\big(p,f,\nh f, \nh^2 f\big) = \Delta_{\cH,\infty} f \overset{def}{=} <\nh^2 f(\nh f),\nh f>.
\end{equation}
Similarly to its Euclidean predecessor, the fully nonlinear operator $\Delta_{\cH,\infty}$ is formally obtained as the limiting case as $q\to \infty$ of the operator $\Delta_{\cH,q}$ in \eqref{hplap}. 
The existence and uniqueness of viscosity solutions for the horizontal $\infty$-Laplacean in Carnot groups was proved in \cite{BC}, and subsequently in \cite{Wang} for general H\"ormander type vector fields, see also \cite{WY}.  
A more probabilistic approach to existence and uniqueness of viscosity solutions for \eqref{infty} is obtained from the theory of tug-of-war games. In Theorems 1.3 and 1.4 in \cite{Tug} the authors proved that in any metric space which is also a length-space (this assumption cannot be removed, see \cite{ACJ}), whenever the boundary datum is Lipschitz and bounded from below (or from above), the continuum value of the game exists and describes the unique absolutely minimizing Lipschitz extension (AMLE) of the datum. By Proposition 4.20 in \cite{Gems} every Carnot-Carath\'eodory space, and therefore in particular any Carnot group, is a length-space. It follows that the results in \cite{Tug} apply, and by  combining them with Theorem 4.10 in \cite{DMV} the continuum value of the tug-of-war  in \cite{Tug}  gives directly the unique Lipschitz viscosity solution of \eqref{infty}  in all sub-Riemannian manifolds, and therefore in particular in Carnot groups.
Thus, similarly to what happens to \eqref{hplap} in the case $q>Q$, viscosity solutions of the problem \eqref{HorCapacity_intro} for the fully nonlinear operator $\Delta_{\cH,\infty}$ do generically satisfy the hypothesis $f\in C(\overline{\Omega})$ in Theorem \ref{teoremaPuntoGenerico} for \emph{any} condenser. 

\medskip
 
Finally, the fully nonlinear operator \eqref{infty} 
fulfills the hypothesis \eqref{proper},
\eqref{DegenerateElliptic}, \eqref{structuralAssumptionF} (Lemma \ref{L:xjhom} and \eqref{laphom}), whereas \eqref{CP} follows from the comparison principle established in \cite{Bieske_infty}, \cite{BC} and also in \cite{Wang}. For a further extension the reader can see \cite{WY}.

\section{Starshaped  sets in Carnot groups}\label{S:star} 

In Definition \ref{D:star} we have introduced the notion of starhapedness in a Carnot group. In \cite{D-G} such notion was first introduced in the following alternative way. Let $\{\delta_\la\}_{\la>0}$ be the non-isotropic group dilations \eqref{dilG}. We denote by $\mathscr Z$ the infinitesimal generator of this group of automorphisms. This vector field acts on a function $f:\G\to \R$ according to the Lie formula
\begin{equation}\label{Z}
\mathscr Z f(p) = \lim_{\lambda\to 1} \frac{f\big(\delta_\lambda (p)\big) - f(p)}{\lambda - 1} ,\quad\quad\quad p\in \G. 
\end{equation}

Using \eqref{dilG} and \eqref{Z}, it is easy to recognise that, if we identify $p = \exp(u)\in \G$ with its logarithmic coordinates $(u_1,...,u_m,u_{2,1},...,u_{2,m_2},...,u_{k,1},...,u_{k,m_k})$, then
the vector field $\Z$ takes the form
\begin{equation}\label{Z1}
\Z = \sum_{j=1}^m u_j \frac{\partial}{\partial u_j} + 2 \sum_{s=1}^{m_2} u_{2,s} \frac{\partial}{\partial u_{2,s}} + ... + k \sum_{\ell=1}^{m_k} u_{k,\ell} \frac{\partial}{\partial u_{k,\ell}}.
\end{equation}
We emphasize here that from the representation \eqref{Z1} and the Baker-Campbell-Hausdorff formula one can prove the following result from \cite{DG2}. For every $j=1,...,k$ denote by $\{e_{j,1},...,e_{j,m_j}\}$ an orthonormal basis of the layer $\cg_j$ of the Lie algebra, keeping in mind that when $j=1$ we write $\{e_1,...,e_m\}$. As in \eqref{xj}, define left-invariant vector fields in $\G$ by setting 
\[
X_{j,\ell}(p) = dL_p(e_{j,\ell}),\ \ \ \ \ \ \ \ \ \ j=1,...,k,\ \ell =1,...,m_j.
\]
Again, we denote by $\{X_1,...,X_m\}$ the family $\{X_{1,1}, ... ,X_{1,m_1}\}$.

\begin{prop}\label{P:Z}
There exist polynomials in $\G$, $Q_{j,\ell}$, $j=1,...,k$, $\ell =1,...,m_j$, where for each $j=1,...,k$ the function $Q_{j,\ell}$ is homogeneous of degree $j$, i.e., $\delta_\lambda Q_{j,\ell} = \lambda^j Q_{j,\ell}$,
such that
\[
\Z = \sum_{\ell=1}^m Q_{1,\ell} X_\ell + \sum_{\ell=1}^{m_2} Q_{2,\ell} X_{2,\ell} + ... + \sum_{\ell=1}^{m_k} Q_{k,\ell} X_{k,\ell}.
\]
\end{prop}

In particular, Proposition \ref{P:Z} shows that $\Z$ involves differentiation not just in the horizontal directions, but in any layer of the stratification of $\cg$. In particular, the vector field $\Z$ is not horizontal.

\medskip

We recall the following result established in \cite{D-G}.

\begin{prop}\label{P:commutator}
In a Carnot group $\G$ with homogeneous dimension $Q$, the infinitesimal generator of group dilations $\Z$ enjoys the following properties:
\begin{itemize}
\item[(i)]
$\operatorname{div}_{\G} \Z \equiv Q$.
\item[(ii)]
Let $\{X_1,...,X_m\}$ be as in \eqref{xj} above. Then, for any $i=1,...,m$ one has
\[
[X_i,\Z] = X_i.
\]
\item[(iii)]
If $f \in C^{3}(\G)$, then
\[
\Delta_\cH (\Z f) = \Z(\Delta_\cH f) + 2 \Delta_\cH f.
\]
\item[(iv)]
In particular, $\Z f$ is $\Delta_\cH$-harmonic, if such is $f$.
\end{itemize}
\end{prop}

The next Euler type result is folklore. Its proof follows along the lines of its Euclidean ancestor.

\begin{lemma}\label{L:hom}
Let $\Om\subset \G$ be a starshaped open set. Then, $f\in C^1(\Om)$ is homogeneous of degree $\kappa\in \mathbb R$ if and only if one has in $\Om$
\[\Z f = \kappa f .\]
\end{lemma}

The following notion alternative to Definition \ref{D:star} above is Definition 3.2 in \cite{D-G}. The inner product is that in a left-invariant Riemannian metric in $\G$.

\begin{defi}\label{D:alt}
Let $\Om\subset \G$ be a bounded $C^1$ domain containing the identity $e$. We say that $\Om$ is \emph{starshaped} if denoting by $\nu$ the (Riemannian) outer normal field to $\Om$, one has on $\p \Om$ 
\begin{equation}\label{znu}
<\Z,\nu> \ge 0.
\end{equation}
When the strict inequality holds in \eqref{znu}, then we say that $\Om$ is \emph{strictly starshaped}.
\end{defi}

We observe that if $\Om$ is a bounded $C^1$ domain, the divergence theorem and (i) in Proposition \ref{P:commutator} imply
\[
\int_{\p \Om} <\Z,\nu> dH_{N-1} = \int_{\Om} \operatorname{div}_\G \Z dp = Q \operatorname{Vol}_\G(\Om),
\]
where we have denoted by $dH_{N-1}$ the $(N-1)-$dimensional Hausdorff  measure in $\G$ restricted to $\p \Om$.
We infer that we cannot have $<\Z,\nu> \le 0$ everywhere on $\p \Om$. Therefore, for any bounded $C^1$ domain
there exist $p_0\in \p \Om$ and an open set $U$ such that $<\Z,\nu> > 0$ at every $p\in \p\Om\cap U$.

The next connection between starshapedeness and the vector field $\Z$ was found in \cite{DG2}.

\begin{prop}\label{P:zstar}
Let $\Om\subset \G$ be a bounded $C^1$ domain. One has the following implications: 
\begin{itemize}
\item[(i)] If $\Om$ is starshaped according to Definition \ref{D:star}, then it is so according to Definition \ref{D:alt}.
\item[(ii)] If $\Om$ is strictly starshaped according to Definition \ref{D:alt}, then it is starshaped according to Definition \ref{D:star}.
\end{itemize}
\end{prop} 

\begin{proof}
Since $\Om$ is $C^1$, given a point $p_0\in \p \Om$ there exists an open set $U$ containing $p_0$ and a local defining function $\rho\in C^1(U)$ such that $\nabla \rho(p)\not= 0$ for every $p\in U$, and for which 
\[
\Om \cap U = \{p\in U\mid \rho(p)<0\},\ \ \ \ \ \ \ \p \Om \cap U = \{p\in U\mid \rho(p) = 0\}.
\]
Therefore, the outer unit normal to $\p \Om$ in $p_0$ is given by $\nu = \nabla \rho(p_0)/|\nabla \rho(p_0)|$. We have
\begin{equation}\label{Znu}
<\Z,\nu> = \frac{<\Z,\nabla \rho(p_0)>}{|\nabla \rho(p_0)|} = \frac{\Z\rho(p_0)}{|\nabla \rho(p_0)|}.
\end{equation}
By \eqref{Z} we know that
\[
\mathscr Z \rho(p_0) = \lim_{\lambda\to 1} \frac{\rho\big(\delta_\lambda (p_0)\big) - \rho(p_0)}{\lambda - 1} = \lim_{\lambda\to 1} \frac{\rho\big(\delta_\lambda (p_0)\big)}{\lambda - 1}. 
\]
(i) If now $\Om$ is starshaped according to Definition \ref{D:star}, then one easily verifies that such is also $\overline{\Om}$. This implies that for some small $\ve>0$  the point $\delta_\la(p_0)$ with $1-\ve\le \la < 1$ belongs to $\Om\cap U$. Therefore, the quotient $\frac{\rho\big(\delta_\lambda (p_0)\big)}{\lambda - 1}$ is $\ge 0$. Passing to the limit as $\la \to 1^-$ we conclude that $\Z\rho(p_0) \ge 0$. By \eqref{Znu}, this implies that $<\Z,\nu> \ge 0$ at $p_0$.

(ii) By hypothesis we know that $<\Z,\nu> \ge 0$ at any point $p_0\in \p \Om$. We want to prove that $\Om$ must be starshaped according to Definition \ref{D:star}. We argue by contradiction and assume that there exist $p\in\Om$ such that for some $0<\bar \la<1$ the point $p_0 \overset{def}{=} \delta_{\bar \la}(p) \not\in \Om$. Let 
\[
\la_1 = \sup \{\la\ge 1\mid \delta_\la(p_0)\not\in \Om\}.
\]
Notice that such sup must be finite since the set of competing $\la$'s is bounded above by ${\bar \la}^{-1} >1$. In fact, we have $\delta_\la(p_0) = \delta_{\la \bar \la}(p) = p\in \Om$, when $\la = {\bar \la}^{-1}$.
We must have $p_1 \overset{def}{=} \delta_{\la_1}(p_0)\in \p \Om$, while $\delta_\la(p_1)\not\in \Om$ for all $\la<1$ sufficiently small. If $\rho$ is a local defining function for $\Om$ at $p_1$ as above, we thus have $\rho(\delta_\la(p_1)) >0$ for such $\la$'s, and therefore
\[
\frac{\rho(\delta_\la(p_1)) - \rho(p_1)}{\la -1} = \frac{\rho(\delta_\la(p_1))}{\la -1} <0.
\]
Passing to the limit for $\la \to 1^-$ we conclude that it must be $\Z\rho(p_1)\le 0$. This contradicts the assumption $<\Z,\nu> > 0$. 
\end{proof}

Let us notice explicitly than if $\rho$ is a (global) defining function for $\Om$, i.e., there exists an open neighborhood $U$ of $\p \Om$ such that $\rho\in C^1(U)$, and 
\[
\Om \cap U = \{p\in U\mid \rho(p)<0\},
\]
then a non-unit outer normal to $\p \Om$ is given by $\nabla \rho$, and thus the condition \eqref{znu} is equivalent to 
\begin{equation}\label{znu2}
\Z \rho = <\Z,\nabla \rho> \ge 0\ \ \ \ \ \ \ \ \text{on}\ \p \Om.
\end{equation}

\medskip

An important property of horizontal Laplaceans in Carnot groups follows from Theorem 2.1 in Folland's seminal paper \cite{Fo}. Such result provides a first fundamental example of starshaped sets, namely the superlevel sets of the fundamental solution. 

\begin{prop}\label{P:folland}
Let $\G$ be a Carnot group and let $\Delta_\cH$ be the horizontal Laplacean associated with an orthonormal basis of the horizontal layer $\cg_1$ of its Lie algebra. Then, the unique positive fundamental solution $E$ of $-\Delta_\cH$ with singularity at $e$ and vanishing at infinity is homogeneous of degree $2-Q$, where $Q$ is the homogeneous dimension of $\G$ as in \eqref{QG}.
\end{prop}

\begin{prop}\label{P:cor}
Let $\G$, $\Delta_\cH$ and $E$ be as in Proposition \ref{P:folland}. Then, all the superlevel sets of the fundamental solution $E$,
\[
\Om_\ell = \{p\in \G\mid E(p)\ge \ell\},\ \ \ \ \ \ \ \ \ \ \ 0<\ell<\infty,
\]
are strictly starshaped and in fact their boundaries are $C^\infty$ compact manifolds of codimension one. 
\end{prop}

\begin{proof}
Lemma \ref{L:hom} and Proposition \ref{P:folland} imply that in $\G\setminus \{e\}$
\begin{equation}\label{ZE}
\Z E = (2-Q) E.
\end{equation}
We observe now that Bony's strong maximum principle in \cite{Bony} implies that $E(p)>0$ for every $p\not= e$. Combining this observation with \eqref{ZE} we see that $\Z E(p) <0$ for every $p\not= e$. This implies that
at no point $p\in \G\setminus\{e\}$ one can have $\nabla E(p) = 0$. Otherwise, by \eqref{Z1} we would have at that point $\Z E(p) = 0$. Since again by Bony's strong maximum principle every superlevel set $\Om_\ell$ is a compact set, we conclude that $\p \Om_\ell$ is a $C^\infty$ compact manifold of codimension one. Finally, $\Z E <0$ in $\G\setminus \{e\}$ implies that $\Om_\ell$ is strictly starshaped.
 
\end{proof}

\medskip

\begin{rem}
We mention that the counterpart of Folland's Proposition \ref{P:folland} for the horizontal $q$-Laplacean operator $\Delta_{\cH,q}$ in \eqref{hplap} represents a challenging open question, see the Open Problem in Section 3.5 in \cite{Gems}. Besides linearity, a crucial tool in the proof of Proposition \ref{P:folland} is H\"ormander's hypoellipticity theorem. For the nonlinear operator $\Delta_{\cH,q}$ one possible way to circumvent these obstructions is by means of $C^{1,\alpha}$ estimates, but as we have indicated above this is not known in general Carnot groups.  
\end{rem}

\medskip

Besides the level sets of the fundamental solution of $\Delta_\cH$ in a Carnot group $\G$ further basic examples of (strictly) starshaped sets are the so-called \emph{gauge balls}. Let $||\cdot ||$ denote the Euclidean distance to the origin in the Lie algebra $\cg$. For $u = u_1 + \cdots + u_k \in \cg$, $u_i\in \cg_i$, one defines
\begin{equation}\label{gauge}
|u|_{\cg} = \left(\sum_{i=1}^k ||u_i||^{2r!/i}\right)^{2r!}.
\end{equation}
The \emph{non-isotropic gauge} in $\G$ is defined by letting 
\begin{equation}\label{nig}
|p|_{\G} = |\exp^{-1} p|_{\cg}, \quad\quad \ \ \ \ \ p\in \G,
\end{equation}
see \cite{Fo} and \cite{FS}. Since the exponential mapping $\exp:\cg \to \G$ is a $C^\infty$ (in fact, $C^\omega$) diffeomorphism, it is clear that $p\to |p|_\G$ is $C^\infty(\G\setminus \{e\})$.
In the sequel we will omit the subscript $\G$, and simply write $|p|$ for the gauge \eqref{nig} in $\G$. One easily recognizes from \eqref{dilg}, \eqref{dilG} and \eqref{nig} that the gauge $|\cdot|$ is homogeneous of degree one, i.e.,  
\begin{equation}\label{homgauge}
|\delta_\lambda p| = \lambda |p|.
\end{equation}
If we thus set $\rho(p) = |p|$, then according to Lemma \ref{L:hom} one has in $\G\setminus \{e\}$
\[
\Z \rho = \rho.
\]
In view of Definition \ref{D:alt} this proves that the gauge ball centered at $e\in \G$ and of radius $r$,  defined by
\begin{equation}\label{gb}
B_\G(e,r) = \{p\in \G\mid |p|<r\},
\end{equation}
is strictly starshaped. By (ii) of Proposition \ref{P:zstar} it is also starshaped according to Definition \ref{D:star}.
However, as we have already mentioned in Section \ref{S:intro} the situation is much subtler than in the Euclidean setting. Even in the prototypical framework of the step two Heisenberg group $\Hn$ there exists a continuum of points in the gauge ball $B_{\Hn}(e,r)$ with respect to which such a set is not starshaped in the sense of Definition \ref{D:star}. This quite surprising phenomenon was first discovered in unpublished 1998 joint work of the second named author with L. Capogna and D. Danielli, and full details will appear in \cite{DG2}. This example shows that the geometric notion of \emph{weak horizontal convexity} in general Carnot groups introduced in Definition 5.5 in \cite{D-G-N} is dramatically different from the standard Euclidean notion of convexity since, for the latter, every convex set is trivially starshaped with respect to any of its points. On the other hand, it was proved in \cite{D-G-N} that the gauge balls in $\Hn$ (or in any group of Heisenberg type) are weakly horizontally convex. Furthermore, these sets are also Euclidean convex and they have real analytic boundary. 

With regard to the theory of weak horizontal convexity developed in \cite{D-G-N} we mention that it was proved in Theorem 5.12 in the same paper that for a function $f$ in the Folland-Stein's class $\Gamma^2(\G)$ one has
\[
f\ \text{is weakly horizontally convex}\ \Longleftrightarrow\  \nabla^2_\cH f \ge 0,
\] 
where $\nabla^2_\cH f$ denotes the symmetrized horizontal Hessian of $f$ as in \eqref{hessian} above (with respect to any orthonormal basis of the horizontal layer $\cg_1$). In this connection we mention that in \cite{ManfrediHeis} the authors independently introduced in the Heisenberg group $\Hn$ a viscosity notion of horizontal convexity, called $v$-\emph{convexity}, which requires that  $\nabla^2_\cH f\ge 0$ in the viscosity sense. It was subsequently proved in \cite{ManfrediCarnot} that:
\[
\text{weak horizontal convexity} \ \Longleftrightarrow\ \text{$v$-convexity}.
\]
For a generalization of the notions of weak horizontal convexity and $v$-convexity to a system of smooth vector fields we refer the reader to \cite{BD1}, \cite{BD2}.  


\section{Proof of the main theorem}\label{S:proof}

In this section we prove our main result, Theorem \ref{teoremaPuntoGenerico}. In a Carnot group $\G$ we consider the problem \eqref{HorCapacity_intro} in a \emph{condenser} $(\Omega_0,\Omega_1)$ given by two connected bounded open sets $\Omega_1\subset \overline \Omega_1 \subset \Omega_0 \subset \G$, and we define $\Omega = \Omega_0 \setminus \overline \Omega_1$. We assume that there exist a point $p_0\in \Om_1$ with respect to which $\Om_1$ and $\Om_0$ are starshaped in the sense of Definition \ref{D:star}.

\medskip

\medskip

We consider the fully nonlinear partial differential equation in $\Omega$
\begin{equation}\label{PDE}
\mathscr F\big(p,f,\nh f(p),\nh^2 f(p)\big)=0.
\end{equation}

Given $f, \vf\in C(\Om)$ and $\bar{p}\in \Om$ we say that $\vf$ \emph{touches} $f$ \emph{from above} at $\bar{p}$ if $\vf(\bar{p}) = f(\bar{p})$ and $\vf(p) \ge f(p)$ for any $p$ in a neighborhood of $\bar{p}$. Similarly, we say that $\vf$ \emph{touches} $f$ \emph{from below} at $\bar{p}$ if $\vf(\bar{p}) = f(\bar{p})$ and $\vf(p) \le f(p)$ for any $p$ in a neighborhood of $\bar{p}$

\begin{defi}\label{Viscosity solutions}
We say that $f\in C(\Om)$ is a 
 \emph{viscosity subsolution} of \eqref{PDE} at the point $\bar{p}\in \Omega$ if for every $\varphi\in C^2(\Omega)$ that touches $f$ from above at  $\bar{p}$ one has
$$
\mathscr F\big (\bar{p},f(\bar{p}),\nh \varphi(\bar{p}),\nh^2 \varphi(\bar p)\big)\geq 0.
$$
We say that $f\in C(\Om)$ is a \emph{viscosity supersolution} of \eqref{PDE} at $\bar{p}\in \Omega$ if for every $\psi\in C^2(\Omega)$ that touches $f$ from below at $\bar{p}$ one has
$$
\mathscr F\big (\bar{p},f(\bar{p}),\nh \psi(\bar{p}),\nh^2 \psi(\bar{p})\big)\leq 0.
$$
\item We say that $f\in C(\Om)$ is a \emph{viscosity solution} of \eqref{PDE} in $\Omega$ if it is both a viscosity subsolution and a viscosity supersolution at any point $\bar{p}\in \Omega$.
\end{defi}

For the notion of viscosity solution in the classical setting of $\R^N$ we refer the reader to the landmark paper \cite{Crandall-LionsUserGuide}, see also the books \cite{Caffarelli} and \cite{Bardi-Capuzzo}. 
For the definition of viscosity solution in the case of the $q$-Laplacean we refer to \cite{tom} (Section 3.3, case $\varepsilon=0$) or  \cite{Bieske_P}. In particular, in Section \ref{SS:infty} we have emphasized that the classes of viscosity and weak solutions of the $q$-Laplacean coincide (see Definition 3 and Lemma 5.5 in \cite{Bieske_P}). As a consequence, in Definition \ref{Viscosity solutions} we do not need to worry about points at which the horizontal gradient of the testing function $\varphi\in C^2(\Omega)$ vanishes. Such points, are dealt with exactly as in the Euclidean case by the notion of weak solution.

In Section \ref{S:models} we have already discussed the existence of viscosity solutions to the problem \eqref{HorCapacity_intro} in the given examples. 
Here we simply note that, whenever constant functions are solutions of \eqref{PDE}, then the comparison principle implies the maximum principle, which gives in particular that 
$
0\leq f\leq 1 $ on $\overline{ \Omega}$, when the boundary conditions in  \eqref{HorCapacity_intro} are in force.

\medskip

We are ready to prove our main result. 

\vskip 0.2in

\begin{proof}[Proof of Theorem \ref{teoremaPuntoGenerico}]

We give the proof only in the case $p_0=e$.  i.e. when the sets $\Om_i$ are starshaped with respect to the group identity; the general case can be obtained with obvious suitable adaptations (essentially consisting in substituting any occurrence of $\delta_{\lambda}$ with $\delta^{p_0}_\lambda$), thanks to the invariance with respect to the left-translations of the vector fields $X_j$.

\medskip

Then let $p_0=e$. For a given $0<\ell<1$ we consider the closed superlevel set of $f$,
 \[
 \Om_\ell = \{p\in \overline \Om_0\mid f(p)\geq\ell\}.
 \]
According to Definition \ref{D:star} we need to prove that, if we define the \emph{starshaped hull} of $\Om_\ell$ as 
\begin{equation}\label{p1}
\Omega^\star_\ell \overset{def}{=} \bigcup_{\la\in (0,1]}\delta_\la\big(\Omega_\ell\big),
\end{equation}
then we have $\Om_\ell = \Omega^\star_\ell$. 

Notice that, since $e\in\Omega_1$ and $\Omega_0$ is bounded, there exists $\Lambda>1$ such that $\Omega_0\subset\delta_\Lambda(\Omega_1)$. Then we consider the following function, named the \emph{starshaped  envelope} of $f$:
\begin{equation}\label{StartshapedFunction}
f^\star(p) \overset{def}{=} \sup \left\{
f\big(\delta_\la(p)\big)\,|\la\in[1,\Lambda],\, p\in \delta_{\la^{-1}}\left(\overline{\Omega}_0
\right)\right\}\quad\text{ for } p\in\overline\Omega_0\,.
\end{equation}
Since both $f$ and the dilations are continuous functions, the supremum in \eqref{StartshapedFunction} is actually a maximum and $f^\star$ is continuous in $\overline\Omega_0$. Given $p\in\overline\Omega_0$, let us call $\overline{\la}=\overline{\la}({p})$ the value that realizes the maximum for $f^\star({p})$, i.e. $\overline{\la}\geq 1$ is such that
$$
f^\star({p})=f(\delta_{\overline{\la}}
({p})).
$$
In general the value at which the maximum is attained may not be unique and in such a case we take the smallest one, that is
\begin{equation}\label{defla}
\overline\lambda(p)=\min\{\overline \lambda\geq 1\,:\, f^\star(p)=f(\delta_{\overline{\la}}
({p}))\}\,.
\end{equation}

Note that  $0\leq f\leq 1$ implies $0\leq f^\star\leq 1$. Furthermore, 
 for every $p \in \overline{\Om}_0$ one has $p = \delta_1(p)$, we infer that the value $\la = 1$ is admissible in the definition of $f^\star(p)$, and thus it is trivially true that
\begin{equation}\label{ff}
 f \le f^\star.
 \end{equation}
Since $f \equiv 1$ on $\overline{\Om}_1$, this immediately yields
 \begin{equation}\label{f*1}
 f^\star=1\quad\text{in }\overline\Omega_1\,.
 \end{equation}
Notice also that for every $p\in\overline\Omega_1$, it holds $\overline\la(p)=1$.
 
Concerning the properties of the starshaped hull and envelope,  here we only need to notice that  $\Omega^\star_\ell$ is the $\ell$ superlevel set for $f^\star$, i.e.,
\begin{equation}\label{stars}
\Omega^\star_\ell=\{p\in \overline \Om_0\mid f^\star(p)\geq \ell\}.
\end{equation}
Indeed, if $p\in\Omega_\ell^\star$, then by \eqref{p1} we have $p\in\delta_\mu(\Omega_\ell)$ for some $\mu\in(0,1]$, i.e. $p=\delta_\mu(q)$ for some $q\in\Omega_\ell$. Set $\la=\mu^{-1}\geq 1$, then we have $q=\delta_\la(p)$, whence definition \eqref{StartshapedFunction} above gives $f^\star(p)\geq f(\delta_\la(p))=f(q)\geq\ell$, which shows 
\[
\Omega^\star_\ell\subseteq\{p\in \overline \Om_0\mid f^\star(p)\geq \ell\}.
\]
Conversely, assume $f^\star(p)\geq\ell$, and let $\overline\la=\overline\la(p)\geq1$ be as in \eqref{defla}. We have $f(\delta_{\overline\la}(p)) = f^\star(p) \geq\ell$, i.e. $q=\delta_{\overline\la}(p)\in\Omega_\ell$. By setting $\mu=\overline\la^{-1}\in(0,1]$, we can write $p=\delta_\mu(q)\in\delta_\mu(\Omega_\ell)$, which proves
\[
\Omega^\star_\ell\supseteq\{p\in \overline \Om_0\mid f^\star(p)\geq \ell\}.
\]
The latter two inclusions establish \eqref{stars}, and since $\Om_\ell^\star$ is starshaped by definition, we infer that $f^\star$ has starshaped level sets. Our final goal is to prove that 
$f^\star=f$.

\medskip

Since \eqref{ff} holds,  we only need to prove the reverse inequality. 
We will reach this conclusion by the following steps:
\begin{itemize}
\item[(a)] we show that $f^\star$ satisfy the same boundary conditions on $\p \Om$ as $f$;
\item[(b)] we prove that $f^\star$ is a viscosity  subsolution of \eqref{HorCapacity_intro};
\item[(c)] using (a) and (b), we then appeal to the hypothesis \eqref{CP} above to conclude that $f^\star \le f$ in $\Om$.
\end{itemize}

(a) First, we verify for $f^\star$ the two boundary conditions in problem of \eqref{HorCapacity_intro}. In view of \eqref{f*1}, we only need to verify the condition $f^\star = 0$ on $ \partial \Omega_0$. For this we observe that the starshapedness of the open set $\Om_0$ implies
\begin{equation}
\label{claimVenerdi}
p\in \partial \Omega_0\quad
\Longrightarrow
\quad
\delta_{\lambda}(p)\in \G\setminus \Omega_0, \; \ \ \forall\, \lambda\geq 1.
\end{equation}
The claim \eqref{claimVenerdi} trivially implies $f^\star(p)=0$ for every $p\in \partial \Omega_0$.
To prove \eqref{claimVenerdi} we use the starshapedness of the open set $\Omega_0$. Assume by contradiction that there exists a $\overline{\la}>1$ such that $\overline{p}=\delta_{\overline{\la}}(p)\in \Omega_0$. Then, since $\Omega_0$ is starshaped w.r.t. the origin,
we have $\delta_{{\mu}}(\overline{p})\in \Omega_0$ for all $0\leq \mu\leq 1$. Using the properties of dilations and noticing that $0<1/\overline{\la}<1$ we obtain
$$
p=\delta_{\overline{\la}}^{-1}(\overline{p})=\delta_{1/\overline{\la}}(\overline{p})\in \Omega_0,
$$
which contradicts $p\in \partial \Omega_0$. Thus \eqref{claimVenerdi}  is proved.\\

(b) Next, we prove that $f^\star$ is a viscosity subsolution for the equation in  \eqref{HorCapacity_intro}. Let $\overline{p}\in \Omega$ be an arbitrary point, and consider a test function $\varphi\in C^2(\Omega)$ such that $\varphi$ touches $f^\star$ from above at $\overline{p}$. We thus have in a sufficiently small (gauge) neighborhood $B$ of $\overline p$
\[
\vf(\overline p) = f^\star(\overline p),\ \ \ \ \ \ \vf(p) \ge f^\star(p)\ \ \ \forall p\in B.
\]
We want to prove that 
\begin{equation}\label{vis}
\mathscr F\left(
\overline{p},f^\star(\overline{p}),\nh\varphi(\overline{p}),
\nh^2\varphi(\overline{p})
\right)\ge 0.
\end{equation}
This shows that $f^\star$ is a viscosity subsolution of \eqref{HorCapacity_intro} at $\overline p$, and thus in $\Om$ by the arbitrariness of $\overline p$.
 
We define
\begin{equation}
\label{Test_fuction}
\psi(p)=\varphi
\left(
\delta^{-1}_{\overline{\la}}(p)\right)
=\varphi\left(\delta_{1/\overline{\la}}
(p)\right),
\end{equation}
where $\overline{\la}=\overline{\la}(\overline{p})$ is as in \eqref{defla}, so that
\begin{equation}\label{ps1}
f^\star(\overline{p})=f(\delta_{{\overline{\la}}}(\overline{p}))\,.
\end{equation}
Clearly $\psi\in C^2(\Omega)$. Moreover, $\psi$ touches $f$ from above at the point 
$$\widetilde{p}=\delta_{\overline{\la}}(\overline{p})\in \Omega.$$
Indeed, first notice that 
$$\psi(\widetilde{p})=\varphi\left(\delta_{1/\overline{\la}}
(\delta_{\overline{\la}}(\overline{p}))\right)=\varphi(\overline{p})=f^\star(\overline{p})= f(\delta_{\overline{\la}}(\overline{p})) = f(\widetilde{p}),$$
where in the second to the last equality we have used \eqref{ps1}.  
Then, for $p\in B$ the point $q=\delta_{\overline{\la}}(p)$ belongs to the neighborhood $\delta_{\overline{\la}}(B)$ of $\widetilde{p}$, we can deduce
\begin{align*}
f(q)-\psi(q)&=f(\delta_{\overline{\la}}(p))-\varphi(\delta_{1/\overline{\la}}(\delta_{\overline{\la}}(p)))
\leq 
f^\star(p)-\varphi(p)\leq f^\star(\overline{p})-\varphi(\overline{p})=0,
\end{align*}
where we have used that $f^\star$ is defined as the supremum of $f(\delta_{\overline{\la}}(p))$ for all $\la\geq 1$, and that $\varphi$ touches $f^\star$ from above at $\overline{p}$.\\

Since $f$ is a viscosity solution of \eqref{PDE} at $\tilde p$, we infer that
\begin{equation}
\label{Jova}
\mathscr F\big(\widetilde{p},f(\widetilde{p}),\nh \psi(\widetilde{p}),\nh^2 \psi(\widetilde{p})\big)\geq 0.
\end{equation}
To conclude we use Lemma \ref{L:xjhom} and the formula \eqref{laphom} combined with the assumption \eqref{structuralAssumptionF}. The former gives 
\begin{align*}
\nh \psi(\tilde p) & = \frac{1}{\overline \la} \nh \vf(\delta_{\frac{1}{\overline \la}}(\tilde p)) = \frac{1}{\overline \la} \nh \vf(\overline p),\\
\nh^2 \psi(\tilde p) & = \frac{1}{{\overline \la}^2} \nh^2 \vf(\delta_{\frac{1}{\overline \la}})(\tilde p) = \frac{1}{{\overline \la}^2} \nh^2 \vf(\overline p).
\end{align*} 
Using these identities in combination with assumption \eqref{structuralAssumptionF}, we find from \eqref{Jova}
\begin{align*}
0
&\leq 
\mathscr F\big(\delta^{-1}_{\overline{\lambda}}(\widetilde{p}),f(\widetilde{p}),
\lambda\,\nh\psi(\widetilde{p}),
\lambda^2\nh^2 \psi(\widetilde{p})\big)\\
&=
\mathscr F\left(
\delta^{-1}_{\overline{\lambda}}(\delta_{\overline{\la}}(\overline{p})),f^\star(\overline{p}),
\lambda\,\frac{1}{\overline{\la}}\nh\varphi(\overline{p}),
\lambda^2\frac{1}{\overline{\la}^2}\nh^2\varphi(\overline{p})
\right)
=
\mathscr F\left(
\overline{p},f^\star(\overline{p}),\nh\varphi(\overline{p}),
\nh^2\varphi(\overline{p})
\right),
\end{align*}
which establishes \eqref{vis}. 

(c) From (a) and (b) the desired conclusion $f^\star \le f$ follows thanks to the assumption \eqref{CP}.

Since $f = f^\star$ in $\Om$, for any $\ell\in (0,1)$ we have
\[
\Omega_\ell= \{p\in \overline{\Om}_0\,|\, f^\star(p) \geq \ell\} = \Om^\star_\ell.
\]
In view of \eqref{p1} this proves the starshapedness of $\Om_\ell$, thus completing the proof.

\end{proof}

\end{document}